\documentclass[12pt]{amsart}

\usepackage{amssymb}
\usepackage{amsmath}
\usepackage{amsthm}
\usepackage{enumerate}
\usepackage{comment}
\usepackage{graphicx}
\usepackage{hyperref}

\usepackage{xcolor}

\theoremstyle{plain}

\newtheorem{theorem}{Theorem}[section]

\newtheorem{lemma}[theorem]{Lemma}
\newtheorem{proposition}[theorem]{Proposition}
\newtheorem{problem}[theorem]{Problem}
\newtheorem{corollary}[theorem]{Corollary}

\newcounter{maintheorem}

\newtheorem{mainth}[maintheorem]{Theorem}

\theoremstyle{definition}

\newtheorem{definition}[theorem]{Definition}

\newtheorem{example}[theorem]{Example}

\newcommand{\D}{\mathrm{D}}

\newcommand{\dist}{\mathrm{dist}}

\renewcommand{\leq}{\leqslant}
\renewcommand{\geq}{\geqslant}

\newcommand{\R}{\mathbb{R}}
\newcommand{\N}{\mathbb{N}}

\newcommand{\conv}{\mathrm{conv}}

\newcommand{\inte}{\mathrm{int}\:}

\renewcommand{\epsilon}{\varepsilon}
\renewcommand{\phi}{\varphi}

\hypersetup{
    pdftitle={},
    pdfauthor={},
    pdfmenubar=false,
    pdffitwindow=true,
    pdfstartview=FitH,
    colorlinks=true,
    linkcolor=blue,
    citecolor=magenta,
    urlcolor=cyan
}

\author[C.A.~De~Bernardi]{Carlo Alberto De Bernardi}

\address{Dipartimento di Matematica per le Scienze economiche, finanziarie ed attuariali, Universit\`a Cattolica del Sacro Cuore, 20123 Milano,Italy}

\email{carloalberto.debernardi@unicatt.it}
\email{carloalberto.debernardi@gmail.com}

\author[J.~Somaglia]{Jacopo Somaglia}
\address{Politecnico di Milano, Dipartimento di Matematica, Piazza Leonardo da Vinci 32, 20133 Milano, Italy.}
\email{jacopo.somaglia@polimi.it}

\subjclass[2020]{Primary 46B03, 	46B20 ; Secondary 52A07}

\keywords{Renorming, ALUR norm, LUR norm, G\^ateaux norm}

\thanks{The research of the authors has been partially
supported by the GNAMPA (INdAM -- Istituto Nazionale di Alta Matematica).
}
\title[Rotund G\^ateaux smooth  norms which are not LUR]{Rotund G\^ateaux smooth norms which are not locally uniformly rotund}

\begin{document}
	
	\begin{abstract} We provide, in every infinite-dimensional separable Banach space, an average locally uniformly rotund (and hence rotund) G\^ateaux smooth renorming which is not locally uniformly rotund. This solves an open problem posed in a recent monograph by A.J. Guirao, V. Montesinos, and V. Zizler.
		
	\end{abstract}

	\maketitle
\section{Introduction}
 Renorming techniques and relations between smoothness and rotundity play a central role in the study of Banach spaces. Very recently A.J.~Guirao, V.~Montesinos and V.~Zizler published a monograph which collects the most important results, techniques and open problems in renorming theory \cite{GMZ}; we refer also to \cite{DGZ} as a main reference in the field and to \cite{DHR1, DHR2, HQ, GR, Q} for very recent progresses in the topic. Several different notions of rotundity of the unit ball of a Banach space has been introduced and widely studied. The most common, that can be considered already classical,  are {\em strict rotundity} (R), {\em uniform rotundity} (UR), and  {\em local uniform rotundity} (LUR). It is well-known and easy-to-prove that local uniform rotundity of the unit ball implies that {\em each point of unit sphere is strongly exposed by any of its supporting functionals} (spaces satisfying this condition are called almost locally uniformly rotund (almost LUR)  in the literature \cite{AGP, BL}), 
which in turns implies that	
	{\em each point of the unit sphere is a denting point} (spaces satisfying this condition are called average locally uniformly rotund (ALUR) in the literature \cite{LLT, Tro85}). Notice that for reflexive spaces the notions of ALUR and almost LUR coincide (see Theorem~\ref{th: ALUR vs almost Lur} below). 
 
  An interesting problem consists in determining  in which situation and to what extent the rotundity properties cited above are distinct or not, in particular, when additional  smoothness properties are assumed.  Despite an extensive literature on the subject, surprisingly enough, the following problem is open (see \cite[p.495, Problem~5]{GMZ}).

 \begin{problem}\label{p: GMZ}
 Does every infinite-dimensional separable Banach space admit a norm that is rotund and G\^ateaux smooth but not LUR? 
 \end{problem}

 \noindent Another related problem arises naturally taking into account the  fact that each almost LUR Fr\'echet smooth norm is automatically LUR, the proof of which easily follows by the \v Smulyan Lemma (see also \cite[Corollary~2.11]{AGP}).

  \begin{problem}\label{p: AGP}
 Is every   almost LUR G\^ateaux smooth space automatically LUR? 
 \end{problem}

\noindent The purpose of our paper is to answer these two questions. The main result we obtained reads as follows (see, Theorem~\ref{th: main theorem} below).
\begin{mainth}\label{th: thm A}
Every infinite-dimensional separable Banach space admits an average locally uniformly rotund (and hence rotund) G\^ateaux smooth equivalent norm $|\cdot|$ which is not locally uniformly rotund.
\end{mainth} 
Theorem \ref{th: thm A} answers in the affirmative Problem \ref{p: GMZ} above, and, since for reflexive spaces the notions of ALUR and almost LUR coincide, it also solves in the negative Problem \ref{p: AGP} above.

 Let us  describe the structure of the paper. Section~\ref{sec: notation} contains some notation, preliminaries, and a brief study of properties ALUR and almost LUR, showing that the two notions do not coincide in general. Section~\ref{sec: main result} is devoted to the proof of our main result Theorem~\ref{th: main theorem}, whose main ingredient is a geometrical construction in the separable Hilbert space (that is actually a bit hidden in the proof itself) combined with a suitable use of Markushevich bases. Finally, in the last section we  lift our construction to some classes of nonseparable Banach spaces and we present some final remarks. For example, we observe that the existence of our renorming $|\cdot|$ can be obtained in each (WCD) space (see Definition~\ref{def: WCGandsons}). The class of (WCD) Banach spaces is also known as Va\v s\'ak, it contains properly the class of weakly compactly generated Banach spaces (in particular it contains each reflexive space) and it shares good renorming properties. Indeed, each (WCD) has a G\^ateaux smooth LUR equivalent norm, we refer to \cite{DGZ,GMZ,HMVZ,CCS,FGMZ} for more information on this class. Finally, we provide some examples of Banach spaces which are not (WCD) that can be renormed by an ALUR G\^ateaux smooth norm which is not LUR.

\section{Notations and preliminaries}\label{sec: notation}
Throughout this  paper, all normed and Banach spaces are real and infinite-dimensional. Let $X$ be a normed space, by $X^*$ we denote the dual space of $X$. By $B_X$, $U_X$, and $S_X$ we denote the closed unit ball, the open unit ball, and the unit sphere of $X$, respectively.
	Moreover, in situations when more than one
	norm on $X$ is considered, we denote by $B_{(X,\|\cdot\|)}$, $U_{(X,\|\cdot\|)}$ and $S_{(X,\|\cdot\|)}$ the closed unit ball, the open unit ball, and the closed unit sphere  with respect to the norm ${\|\cdot\|}$, respectively. By $\|\cdot\|^*$ we denote the dual norm of $\|\cdot\|$.
	Given a set $A\subset X$ we denote by $\partial A$  the boundary of $A$. For $x,y\in X$, $[x,y]$ denotes the closed segment in $X$ with
endpoints $x$ and $y$, and $(x,y)=[x,y]\setminus\{x,y\}$ is the
corresponding ``open'' segment; the segment $[x,y)$ is defined similarly. We recall a geometric observation, which turns out to be very useful in the next sections (cf. \cite{Klee}).
\begin{lemma}\label{l: segmento interno}
    Let $X$ be a normed space. Let $C$ be a nonempty closed convex subset of $X$ with non-empty interior. If $x\in\inte C$ and $y \in \partial C$, then $[x,y)\subset \inte C$.
\end{lemma}
 A \emph{biorthogonal system} in a Banach space $X$ is a system $(e_\gamma; f_\gamma)_{\gamma\in\Gamma}\subset X\times X^*$, such that $f_\alpha(e_\beta)=\delta_{\alpha,\beta}$ ($\alpha,\beta\in \Gamma$). A biorthogonal system is \emph{fundamental} if ${\rm span}\{e_\gamma\}_{\gamma\in\Gamma}$ is dense in $X$; it is \emph{total} when ${\rm span}\{f_\gamma\} _{\gamma\in\Gamma}$ is $w^*$-dense in $X^*$. A \emph{Markushevich basis} (M-basis) is a fundamental and total biorthogonal system. We refer to \cite{HMVZ} and \cite{S2} for good references on M-bases and \cite{Ha19,HRST,K4} for some recent progresses in the topic.

Let us recall that the duality map $\D_X: S_X\to2^{S_{X^*}}$ is the function defined, for each $x\in S_X$, by
$$\D_X(x):=\{x^*\in S_{X^*};\, x^*(x)=1\}.$$
Given a subset $K$ of $X$, a {\em slice} of $K$ is a
set of the form
$$S(K,x^*,\alpha) := \{x^*\in K;\,  x^*(x)> \sup x^*(K)-\alpha\},$$
where $\alpha> 0$ and $x^*\in X^*\setminus\{0\}$ is bounded above on
$K$. For convenience of the reader, we recall some standard notions which we list in the following definition.

\begin{definition}\label{d: basic def}
Let $x\in S_X$ and $x^*\in S_{X^*}$. We say that:
\begin{itemize}
\item  $x$ is  an {\em extreme point} of $B_X$ if it does not lie in any ``open'' segment contained in $B_X$; 
\item $X$ is rotund (R) if each point of $S_X$ is an extreme point of $B_X$;
\item $x$ is  {\em supported} by $x^*$ if $x^*\in \D_X(x)$;
\item  $x$ is a {\em denting point} of $B_X$ if for each neighbourhood $V$ of $x$ in the norm topology 
there exists a slice $S$ of $B_X$ such that
$x\in S\subset V$;
\item $x$ is  {\em strongly exposed} by $x^*$ if $x^*\in \D_X(x)$
and, for each norm neighbourhood $V$ of $x$, there exists $\alpha>0$ such that $S(B_X,x^*,\alpha)\subset
V$
(equivalently, if $x^*\in \D_X(x)$
and $x_n\to x$ for all sequences $\{x_n\}_n\subset B_X$ such that $\lim_{n\to \infty}x^*(x_n)=1$)
;
\item $x$ is a {\em locally uniformly rotund (LUR) point} of $B_X$
	if 
 $x_n\to x$, whenever $\{x_n\}_{n}\subset X$ is such that $$\lim_n 2\|x\|^2 + 2\|x_n\|^2 - \|x+x_n\|^2
 =0;$$
\item $x$ is a {\em weakly locally uniformly rotund (WLUR) point} of $B_X$
	if 
 $x_n\to x$ weakly, whenever $\{x_n\}_{n}\subset X$ is such that $$\lim_n 2\|x\|^2 + 2\|x_n\|^2 - \|x+x_n\|^2
 =0;$$
\item 
$x$ is a  {\em G\^ateaux smooth point} (respectively, {\em Fr\'echet smooth point})  of $B_X$ if, 
 the following limit
holds $$\lim_{h\to0}\frac{\|x+hy\|+\|x-hy\|-2}{h}=0,$$
whenever $y\in S_X$ (respectively, uniformly for  $y\in S_X$).
\end{itemize}
\end{definition}

If  each point of $S_X$ is LUR, G\^ateaux smooth, Fr\'echet smooth, respectively, we say that $X$ and the corresponding norm $\|\cdot\|$ is  LUR, G\^ateaux smooth, Fr\'echet smooth, respectively.

A Banach space $X$ has the \textit{Kadec property} (resp. \textit{Kadec-Klee property}) if the norm and the weak topology coincide on the unit sphere $S_X$ (resp. convergent sequences in $(S_X,w)$ are convergent in $(S_X,\|\cdot\|)$). If $\ell_1$ does not embed into $X$, then Kadec-Klee implies Kadec, see \cite{Tro85}.

We say that a Banach space $X$ is \textit{average locally uniformly rotund} (ALUR, in short) if every point of the unit sphere $S_X$ is a denting point. In fact, the original definition of ALUR space, given by S. Troyanski in \cite{Tro85}, requires some technical preliminaries which are not needed in our paper, see \cite{LLT} for the equivalence of the two definitions. Clearly, if a norm is LUR, then it is ALUR. On the other hand, the other implication does not hold, see \cite[Example~2.2]{Tro85}. Notice that such an example, provided in $X=\ell_1$, is not G\^ateaux differentiable at $x=e_1$. Furthermore, in \cite{Tro85} it is proved that a norm is ALUR if and only if it is rotund and has the Kadec property. Finally, we recall that in \cite{YOST} it is proved that if $X$ is reflexive, and $\|\cdot\|$ is rotund and has the Kadec property (equivalently the Kadec-Klee property), then $\|\cdot\|^*$ is Fr\'echet smooth. Since we will make use of these two results, we summarize them in the following theorem.

\begin{theorem}\label{th: ALURsseR+KK}
Let $X$ be a Banach space. Then:
\begin{enumerate}
    \item $X$ is ALUR if and only if $X$ is rotund and has the Kadec property;
    \item Suppose that $X$ is reflexive. If $X$ is rotund and has the Kadec property, then $X^*$ is Fr\'echet smooth. 
\end{enumerate}    \end{theorem}

We conclude this section by comparing the definition of average locally uniformly rotund norm and the one of almost locally uniformly rotund norm. A Banach space $X$ is \textit{almost locally uniformly rotund} (\textit{almost LUR}, for short) if for every $x\in S_X$ and every pair of sequences $\{x_n\}_{n}\subset S_X$ and $\{x_n^*\}_{n}\subset S_{X^*}$ such that $\lim_m(\lim_n(x_m^*((x_n+x)/2)))=1$, $\{x_n\}_{n}$ converges to $x$. In \cite{BL} it is proved that a Banach space is almost LUR if and only if each point $x\in S_X$ is strongly exposed by $x^*\in S_{X^*}$ whenever $x$ is supported by $x^*$. This equivalence is used in \cite{AGP} for proving that an almost LUR, Fr\'{e}chet smooth norm is actually LUR. Notice that the same result is no longer true if we replace Fr\'{e}chet by G\^ateaux (combine Theorem \ref{th: main theorem} with Theorem \ref{th: ALUR vs almost Lur} below). In literature almost LUR norms are abbreviated as ALUR norms. Since this choice might be ambiguous, and we were not able to find references that compare the two notions, we decided to include here some results that clarify the relations between these two definitions. 

\begin{theorem}\label{th: ALUR vs almost Lur}
    Let $X$  be a Banach space, then the following holds true
    \begin{enumerate}
    \item if $X$ is almost LUR, then $X$ is ALUR.
    \item if $X$ is reflexive and ALUR, then $X$ is almost LUR.
    \end{enumerate}
\end{theorem}

\begin{proof}
    (i) It follows by the fact that if $x\in S_X$ is strongly exposed by $x^*\in X^*$, then $x$ is a denting point of $B_X$.\\
        (ii) By Theorem \ref{th: ALURsseR+KK} (i) we get that $X$ is rotund and has the Kadec property. Therefore, by applying Theorem \ref{th: ALURsseR+KK} (ii), $X^*$ is Fr\'echet smooth. Let $x\in S_X$ and $x^* \in \D_X(x)$. By the \v Smulyan Lemma, $x$ is strongly exposed by $x^*$, hence $X$ is almost LUR.
\end{proof}

Finally, we provide an equivalent norm on $\ell_1$ that is ALUR but not almost LUR. Our example is based on \cite[Example 2.2]{Tro85}, in which it is proved that there exists an ALUR norm which is not LUR. We just show that such a norm is not almost LUR.

\begin{example}
Let $X=\ell_1$ endowed with the norm 
\begin{equation*}
    \|(x_n)_n\|=\sum_{n=1}^{\infty} |x_n| + \left(\sum_{n=1}^{\infty}\frac{x_n^2}{n^2}\right)^{1/2}.
\end{equation*}
The equivalent norm $\|\cdot\|$ is ALUR, see \cite[Example 2.2]{Tro85}. Let us show that $\|\cdot\|$ is not almost LUR. Let $x:=\frac{e_1}{2}\in S_X$, where $\{e_n\}_n$ is the standard Schauder basis of $\ell_1$. Let $x^*=(x^*_n)_n\in (\ell_{\infty},\|\cdot\|^*)$ be defined by
\begin{equation*}
    x^*_n=\begin{cases}
2 \,\,\, \text{ if } n=1,\\
1 \,\,\, \text{ otherwise}.
    \end{cases}
\end{equation*}
We claim that $\|x^*\|^*=1$. Indeed, let $y=(y_n)_n\in S_X$. We have
\begin{equation*}
    |x^*(y)|\leq 2|y_1| + \sum_{n=2}^{\infty}|y_n|=\sum_{n=1}^{\infty}|y_n| + |y_1| \leq \|y\|.
\end{equation*}
Moreover we have $x^*(x)=1$. Therefore, $\|x^*\|^*=1$ and $x\in S_X$ is supported by $x^*$. It remains to show that $x$ is not strongly exposed by $x^*$. Let $\delta>0$. For sufficiently large $n\in \N$, both $x$ and $(\frac{n}{1+n})e_n\in S_X$ belongs to the slice $S(B_X,x^*,\delta)$. Observing that $\|x - (\frac{n}{1+n})e_n\|>1/2$, we get that $x$ is not strongly exposed by $x^*$, hence $\|\cdot\|$ is not almost LUR.
\end{example}

\section{Main results}\label{sec: main result}

This section is devoted to construct an equivalent norm in any separable infinite-dimensional Banach space, which is ALUR, G\^ateaux smooth but not LUR. Before defining such a norm, we need some preliminary constructions and technical assumptions.
In the sequel of this section, $X$ denotes a separable Banach space. We shall need the following result, that essentially collects some well-known results about renorming of separable Banach spaces and existence of M-bases with additional properties. For the sake of completeness we include a sketch of the proof.

\begin{theorem}\label{th: TroPel}
               There exist an equivalent norm $|\!|\!|\cdot|\!|\!|$ and an  M-basis   $(e_n, g_n)_{n\in\N}$ on $X$ such that:
    \begin{enumerate}
        \item $|\!|\!|\cdot|\!|\!|$ is LUR and G\^ateaux smooth;
        \item for every $x\in X$ we have\begin{equation}\label{eq: non ellisse in e_1}
|\!|\!|x|\!|\!|^2=|\!|\!|x- g_1(x)e_1|\!|\!|^2 + [g_1(x)]^2; 
\end{equation}
\item $|\!|\!|e_n|\!|\!|=1$, whenever $n\in\N$;
\item $|\!|\!|g_{1}|\!|\!|^*=|\!|\!|g_{3n}|\!|\!|^*=1$, whenever $n\in\N$.
    \end{enumerate}  
\end{theorem}

\begin{proof}
Let $|\!|\!|\cdot|\!|\!|_1$ be a  LUR and G\^ateaux smooth equivalent norm on $X$ (see e.g. \cite[Theorem~8.2]{FHHMZ}). By \cite[Proposition~8.13]{S2}, 
there exists an   M-basis   $(e_n, g_n)_{n\in\N}$ on $X$ such that:
    \begin{itemize}
\item $|\!|\!|e_n|\!|\!|_1=1$, whenever $n\in\N$;
\item $|\!|\!|g_{3n}|\!|\!|_1^*=1$, whenever $n\in\N$.
    \end{itemize}
    Let us define an equivalent norm $|\!|\!|\cdot|\!|\!|$ on $X$ by  $$|\!|\!|x|\!|\!|^2=|\!|\!|x- g_1(x)e_1|\!|\!|_1^2 + [g_1(x)]^2,\qquad x\in X.$$
It remains to prove (i)-(iv). Since $|\!|\!|\cdot|\!|\!|_1$ is G\^ateaux smooth then by definition $|\!|\!|\cdot|\!|\!|$ is G\^ateaux smooth too. Let us prove that $|\!|\!|\cdot|\!|\!|$ is  LUR: let $\{x_n\}_{n}\subset X$ and $x\in X$ be such that $\lim_{n} 2|\!|\!|x|\!|\!|^2 + 2 |\!|\!|x_n|\!|\!|^2 - |\!|\!|x + x_n|\!|\!|^2=0$, in other words
\begin{equation*}
\begin{split}
    \lim_{n}\,\,\,& 2|\!|\!|x - g_1(x)e_1|\!|\!|_1^2 + 2 [g_1(x)]^2 + 2|\!|\!|x_n - g_1(x_n)e_1|\!|\!|_1^2 + 2 [g_1(x_n)]^2 \\&- |\!|\!|x + x_n - g_1(x + x_n)e_1|\!|\!|_1^2 + [g_1(x + x_n)]^2=0.
\end{split}
\end{equation*}
Since 
\begin{equation*}
\begin{split}
     &2|\!|\!|x - g_1(x)e_1|\!|\!|_1^2 + 2|\!|\!|x_n - g_1(x_n)e_1|\!|\!|_1^2- |\!|\!|x + x_n - g_1(x + x_n)e_1|\!|\!|_1^2\geq 0,\\
     & 2 [g_1(x)]^2 + 2 [g_1(x_n)]^2  - [g_1(x + x_n)]^2\geq 0,
\end{split}
\end{equation*}
we have
\begin{equation*}
\begin{split}
     &\lim_n\,\,\,2|\!|\!|x - g_1(x)e_1|\!|\!|_1^2 + 2|\!|\!|x_n - g_1(x_n)e_1|\!|\!|_1^2- |\!|\!|x + x_n - g_1(x + x_n)e_1|\!|\!|_1^2= 0,\\
     &\lim_n\,\,\, 2 [g_1(x)]^2 + 2 [g_1(x_n)]^2  - [g_1(x + x_n)]^2= 0,
\end{split}
\end{equation*}
Therefore, by using the fact that $|\!|\!|\cdot|\!|\!|_1$ is LUR, we obtain $$\lim_n [x_n- g_1(x_n)e_1]=x- g_1(x)e_1,\qquad \lim_n g_1(x_n)=g_1(x).$$ Hence $\lim_n|\!|\!|x- x_n|\!|\!|=0$, and (i) follows.\\
For $x\in X$, we have 
\begin{equation*}
\begin{split}
|\!|\!|x - g_1(x)e_1|\!|\!|^2&=|\!|\!|x- g_1(x)e_1 - g_1(x-g_1(x)e_1)e_1|\!|\!|_1^2 + [g_1(x-g_1(x)e_1)]^2\\
&=|\!|\!|x - g_1(x)e_1|\!|\!|^2_1,
\end{split}
\end{equation*}
so (ii) follows. (iii) is trivial. For (iv), let $x\in X$ be such that $|\!|\!|x|\!|\!|=1$. Then we have $1=|\!|\!|x|\!|\!|^2=|\!|\!|x- g_1(x)e_1|\!|\!|_1^2 + [g_1(x)]^2$, from which we get $$|g_1(x)|\leq \sqrt{1-|\!|\!|x- g_1(x)e_1|\!|\!|_1^2}\leq 1.$$ Combining (iii) with $g_1(e_1)=1$, we get $|\!|\!|g_1|\!|\!|^*=1$. Finally, let $n\in \N$ and $x\in X$ be such that $|\!|\!|x|\!|\!|\leq 1$, then, since $|\!|\!|g_{3n}|\!|\!|_1^*=1$, we have
\begin{equation*}
    \frac{|g_{3n}(x)|^2}{|\!|\!|x|\!|\!|^2}=\frac{|g_{3n}(x-g_1(x)e_1)|^2}{|\!|\!|x-g_1(x)e_1|\!|\!|^2_1+[g_1(x)]^2}\leq \frac{|g_{3n}(x-g_1(x)e_1)|^2}{|\!|\!|x-g_1(x)e_1|\!|\!|^2_1}\leq 1.
\end{equation*}
Combining (iii) with $g_{3n}(e_{3n})=1$ we obtain (iv).
\end{proof}

In the sequel of this section, let us denote by $|\!|\!|\cdot|\!|\!|$ and $(e_n, g_n)_{n\in\N}$ the equivalent norm and the M-basis on $X$, respectively, given by Theorem~\ref{th: TroPel}. 
Let us denote by $P_1$  the bounded projection onto the one dimensional space generated by $e_1$ defined by $P_1(x)=g_1(x)e_1$ ($x\in X$), and put $Q_1:=I - P_1$.
Then, \eqref{eq: non ellisse in e_1} implies that $Q_1 [B_{(X,|\!|\!|\cdot|\!|\!|)}]\subset B_{(X,|\!|\!|\cdot|\!|\!|)}$.

We consider the linear operator $T\colon (\ell_2,\|\cdot\|_2) \to (X,|\!|\!|\cdot|\!|\!|) $ defined by 
\[  T\alpha=\sqrt{2}\alpha_1e_1+\sum_{n=2}^{\infty} \frac{1}{n^2} \alpha_n e_n,\]
where $\alpha=(\alpha_n)_n\in\ell_2$. Observing that $\sum_{n=2}^\infty \frac{1}{n^2}<1$, we have $T[R_1B_{\ell_2}]\subset U_{(X,|\!|\!|\cdot|\!|\!|)}$, where $R_1B_{\ell_2}:=\{\alpha=(\alpha_n)_n\in B_{\ell_2}\colon \alpha_1=0\}$. We notice that the operator $T$ is well-defined, bounded, linear, one-to-one, and the range $Y:=T\ell_2$ contains $\{e_n\}_{n}$, therefore $Y$ is dense in $X$. By the injectivity of the operator $T$, we can consider the subspace $Y$ endowed with the norm $\|\cdot\|_{\theta}$ defined by $\|y\|_{\theta}:=\|T^{-1}y\|_2$, for any $y\in Y$. In this way, we obtain that $T$ is an isometric isomorphism between $(\ell_2,\|\cdot\|_2)$ and $(Y,\|\cdot\|_{\theta})$. We set $B:=T[B_{\ell_2}]$. In other words, we have that
\begin{equation*}
    B=\{y\in Y\colon \|y\|_{\theta}\leq 1\}.
\end{equation*}
Similarly, we define
\begin{align*}
    U=\{y\in Y\colon \|y\|_{\theta}< 1\},\\
    S=\{y\in Y\colon \|y\|_{\theta}= 1\}.
\end{align*}
We claim that the convex subset $B$ is compact in $(X,|\!|\!|\cdot|\!|\!|)$. Indeed, the operator $T$ is bounded, therefore it is \textit{w-w}-continuous, hence $B$ is closed in $(X,|\!|\!|\cdot|\!|\!|)$. In order to prove the claim, it is enough to observe that $T$ is a compact operator. Indeed, $T$ is the limit in $\mathcal{K}(\ell_2, X)$ of the sequence of finite-rank operators $\{T_k\}_k$, defined by
\begin{equation*}
T_k \alpha= \sqrt{2}\alpha_1 e_1 +\sum_{n=2}^k\frac{1}{n^2}\alpha_n e_n,
\end{equation*}
where $\mathcal{K}(\ell_2, X)$ is the space of all compact operators from $\ell_2$ to $X$, which is a closed subset of the space of all bounded operators from $\ell_2$ to $X$ (see e.g. \cite[Proposition~1.40]{FHHMZ}). 
By the claim, $B$ is a convex compact set in $(X,|\!|\!|\cdot|\!|\!|)$ and hence we have that the set  $$D=\conv\bigl(B_{{(X,|\!|\!|\cdot|\!|\!|)}}\cup B\bigr)$$
	is closed in ${{X}}$.  
	Then by our definition and by symmetry,  $D$ is the closed unit ball of an equivalent norm $\|\cdot\|$ on ${X}$. Unfortunately such a norm is not rotund, therefore we need a further step. Define $f_n={g_n}/{|\!|\!|g_{n}|\!|\!|^*}$ ($n\in\N$), and consider the equivalent norm $|\cdot|$ on ${X}$ defined  by
\begin{equation}\label{eq: definizione norma}	
 F\bigl(x\bigr):=|x|^2=\|x\|^2+\sum_{n=2}^\infty2^{-n}f_n(x)^2.
 \end{equation}
	 In order to prove the main theorem of this section we need several technical lemmas that describe the geometry of the norm $|\cdot|$. The next results are stated by using the same notation as in the first part of this section.

	\begin{lemma}\label{lemma: U nell'interno}
		We have $U\cap \partial D=\emptyset$, equivalently $B\cap\partial D\subset S$, equivalently $U\subset \inte D$.	
	\end{lemma}

\begin{proof}
		Suppose on the contrary that there exists $x\in U\cap \partial D$.  Combining that $T[R_1B_{\ell_2}]\subset U_{(X,|\!|\!|\cdot|\!|\!|)}\subset \inte D$ with the fact that $\|x\|_{\theta}=\|T^{-1}x\|_2<1$, there exists $y_1\in \R$ such that:
  \begin{itemize}
      \item $y:=y_1 e_1 + Q_1x\in S$ (which is equivalent to $\|y\|_{\theta}=1$);
      \item $x\in [Q_1x,y)$.
  \end{itemize} Since $y\in S\subset D$ and $Q_1x\in\inte D$, by Lemma \ref{l: segmento interno} we get $x\in [Q_1x,y)\subset \inte D$, that is a contradiction. 
	\end{proof}

\begin{lemma}\label{lemma: bordo da bordi} Let $x\in\partial D$. Then there exist $\lambda\in[0,1]$, $b\in S_{{(X,|\!|\!|\cdot|\!|\!|)}}$, and $c\in S$ such that $x=\lambda b+(1-\lambda)c$.
	\end{lemma}
	
	\begin{proof} Let us observe that $D$ is the union of the following sets: 
		$$\conv(S_{{(X,|\!|\!|\cdot|\!|\!|)}}\cup S),\quad U_{{(X,|\!|\!|\cdot|\!|\!|)}},\quad U,$$
		$$\bigcup_{\mu\in(0,1)}[\mu B_{{(X,|\!|\!|\cdot|\!|\!|)}}+(1-\mu)U],\quad  \bigcup_{\mu\in(0,1)}[\mu U_{{(X,|\!|\!|\cdot|\!|\!|)}}+(1-\mu)B].$$
		We observe that  $U_{{(X,|\!|\!|\cdot|\!|\!|)}}\subset \inte D$, by Lemma~\ref{lemma: U nell'interno}, $U\subset \inte D$, and by Lemma \ref{l: segmento interno},  the sets
		$$\quad \bigcup_{\mu\in(0,1)}[\mu B_{{(X,|\!|\!|\cdot|\!|\!|)}}+(1-\mu)U],\quad  \bigcup_{\mu\in(0,1)}[\mu U_{{(X,|\!|\!|\cdot|\!|\!|)}}+(1-\mu)B].$$
		are contained in $\inte D$. Hence necessarily $x\in \conv(S_{{(X,|\!|\!|\cdot|\!|\!|)}}\cup S)$, and the conclusion holds. 
	\end{proof}

\begin{lemma}\label{lemma: no segmenti orizzontali}
 If $x\in\partial D$, then at least one of the following conditions is satisfied:
		\begin{enumerate}
			\item $x+te_1\not\in D$, whenever $t>0$;
			\item $x-te_1\not\in D$, whenever $t>0$.
		\end{enumerate}
In other words, the set $\partial D$ does not contain segments parallel to $e_1$. 
	\end{lemma}
	
	\begin{proof} 
		Let $x\in\partial D$ and suppose without any loss of generality that $g_1(x)\geq 0$, the case $g_1(x)\leq 0$ is similar.
		Let us observe that, by \eqref{eq: non ellisse in e_1}, we get 
		\begin{equation*}\label{eq: P1}
			Q_1\bigl(B_{{(X,|\!|\!|\cdot|\!|\!|)}}\setminus Q_1(B_{{(X,|\!|\!|\cdot|\!|\!|)}})\bigr)\subset U_{{(X,|\!|\!|\cdot|\!|\!|)}}.
   \end{equation*}
   Moreover, since $Q_1 B= T[R_1B_{\ell_2}]$, we have
    \begin{equation*}   
   Q_1(B)\subset U_{{(X,|\!|\!|\cdot|\!|\!|)}}.
		\end{equation*}
		Hence {$Q_1\bigl(D\bigr)\subset B_{{(X,|\!|\!|\cdot|\!|\!|)}}$ and}  $Q_1\bigl(D\setminus Q_1(B_{{(X,|\!|\!|\cdot|\!|\!|)}})\bigr)\subset U_{{(X,|\!|\!|\cdot|\!|\!|)}}$. Let us denote $z=Q_1(x){\subset B_{{(X,|\!|\!|\cdot|\!|\!|)}}}$ and let us consider the following two cases.
		
		\noindent {\bf Case 1: $|\!|\!|z|\!|\!|=1$.} In this case $x=z$ and, since $Q_1\bigl(D\setminus Q_1(B_{{(X,|\!|\!|\cdot|\!|\!|)}})\bigr)\subset U_{{(X,|\!|\!|\cdot|\!|\!|)}}$, we have that 
		(i) is satisfied.
		
		\noindent {\bf Case 2: $|\!|\!|z|\!|\!|<1$.} First, observe that $z\in\inte D$. Suppose on the contrary that (i) does not hold and let $t>0$ be such that $w:=x+te_1\in D$. Then $x\in[z,w)$ and hence $x\in\inte D$, a contradiction.
  
  Let us conclude the proof by observing that $\partial D$ does not contain segment parallel to $e_1$. Indeed, suppose that $[x,x+te_1]\subset \partial D$ for some $t> 0$, and $x\in \partial D$, then the point $x+ \frac{t}{2}e_1 \in \partial D$ does not satisfy neither (i) nor (ii). Which is a contradiction. 
	\end{proof}

The previous lemmas are necessary for proving that the norm defined in \eqref{eq: definizione norma} is rotund.
	
	\begin{proposition}\label{prop: rotund}
		The norm $|\cdot|$ is rotund.	
	\end{proposition}
	
	\begin{proof}
		In order to show that $|\cdot|$ is a rotund norm, by \cite[Fact 7.7]{FHHMZ} it is enough to show that if $x,y\in X$ satisfy
\begin{equation}\label{eq: Qonebar}
    2|x|^2 + 2|y|^2 - |x+y|^2=0,
\end{equation}
        then $x=y$. Suppose by contradiction that there are $x,y\in X$, $x\neq y$ that satisfy \eqref{eq: Qonebar}. 
        Observe that, if $p$ is a seminorm on $X$,  we clearly  have 
\begin{equation}\label{eq: seminorm}
2p^2(x) + 2p^2(y) - p^2(x+y)
\geq[p(x)-p(y)]^2\geq0
\end{equation}
        Then, by the definition of $|\cdot|$ 
        we get
\begin{equation}\label{eq: Qtwobars}
2\|x\|^2 + 2\|y\|^2 - \|x+y\|^2=0,
\end{equation}
and
\begin{equation}\label{eq: Qp2}    
    2f_n^2(x)+ 2 f_n^2(y) - (f_n(x+y))^2=0 \qquad\qquad (n\geq 2).
\end{equation}

Now, combining \eqref{eq: Qtwobars} and \eqref{eq: seminorm}, we obtain 

      \begin{equation}\label{eq: twobarsR}
\|x\|=\|y\|=\bigg\|\frac{x+y}{2}\bigg\|.
      \end{equation} 
      Moreover,  by \eqref{eq: Qp2} and the linearity of $f_n$, it follows that $f_n(x)=f_n(y)$ for every $n\geq 2$.       
		 Since $x\neq y$, $(e_n,g_n)_n$ is an M-basis, and $f_n(x)=f_n(y)$ for $n\geq 2$,  we have that $f_1(x)\neq f_1(y)$. Hence, by \eqref{eq: twobarsR}, $\partial D$ contains a nontrivial segment parallel to $e_1$, 
  which contradicts Lemma~\ref{lemma: no segmenti orizzontali}. Therefore $|\cdot|$ is strictly convex.
	\end{proof}

	\begin{proposition}\label{prop: smooth}
		The norm $\|\cdot\|$ is G\^ateaux smooth.	
	\end{proposition}

	\begin{proof} By the \v Smulyan Lemma, it is sufficient to prove that if $x\in S_{{(X,\|\cdot\|)}}=\partial D$ then there exists a unique $f\in S_{{(X^*,\|\cdot\|^*)}}$ supporting $B_{{(X,\|\cdot\|)}}=D$ at $x$. 
		
		Let 	$x\in\partial D$ and suppose on the contrary that there exist two distinct functionals $h_1,h_2\in S_{{(X^*,\|\cdot\|^*)}}$ such that $h_1(x)=h_2(x)=1$. By Lemma~\ref{lemma: bordo da bordi}, there exist $\lambda\in[0,1]$, $b\in S_{{(X,|\!|\!|\cdot|\!|\!|)}}$, and $c\in S$ such that $x=\lambda b+(1-\lambda)c$.
		Let us consider the following two cases.
		
		\noindent {\bf Case 1: $\lambda>0$.} In this case we proceed as in \cite[Theorem~1.5]{KVZ}: observe that if we define $C_\lambda=\lambda B_{{(X,|\!|\!|\cdot|\!|\!|)}}+(1-\lambda)c$, then $x\in\partial C_\lambda$ and $C_\lambda\subset D$. Hence, $h_1$ and $h_2$ are  functionals supporting $C_\lambda$ at $x$, and we get a contradiction since $|\!|\!|\cdot|\!|\!|$ is  G\^ateaux smooth.

		\noindent {\bf Case 2: $\lambda=0$.} In this case $x=c\in S$ and $h_i(x)=1=\sup h_i(B)$ ($i=1,2$), moreover $h_1|_Y\neq h_2|_Y$, since $ Y$ is a dense subspace of ${{X}}$. Hence,  $h_1|_Y$ and $h_2|_Y$ are distinct $\|\cdot\|_\theta$-continuous functionals (since bounded on $B$) on $Y$ supporting $B$ at $x$, moreover $\|h_1|_Y\|_\theta^*=\|h_2|_Y\|_\theta^*=1$. A contradiction, since $B$ is the closed unit ball of  $(Y,\|\cdot\|_\theta)$, which is a Hilbert space (and hence a G\^ateaux smooth Banach space). 	\end{proof}

\begin{proposition}\label{prop: KP}
		The norm $\|\cdot\|$ has the Kadec property.	
	\end{proposition}
	
	\begin{proof}
	Suppose on the contrary that there exists a net $\{x_\xi\}_{\xi\in\Sigma}\subset S_{(X,\|\cdot\|)}$ that is $w$-convergent to an element $x\in S_{(X,\|\cdot\|)}$ and such that $\|x_\xi-x\|$ is bounded away from 0. 
	By Lemma~\ref{lemma: bordo da bordi}, for each $\xi\in\Sigma$, there 
	exist
	$\lambda_\xi\in[0,1]$, $b_\xi\in S_{(X,|\!|\!|\cdot|\!|\!|)}$, and $c_\xi\in S$ such that $x_\xi=\lambda_\xi b_\xi+(1-\lambda_\xi)c_\xi$. By compactness of $B$ and $[0,1]$, we can suppose that $\lambda_\xi\to\lambda\in[0,1]$ and that  $\{c_\xi\}_{\xi\in\Sigma}$ converges in norm to  $c\in B$.  
		Let us consider the following  two cases.
	
	\noindent {\bf Case 1: $\lambda=0$.}
	In this case, since $\{b_\xi\}_{\xi\in\Sigma}$ is bounded and $\lambda_\xi\to 0$, we have that the net $\{x_\xi\}_{\xi\in\Sigma}$ converges in norm to $c=x$.
 
	\noindent {\bf Case 2: $\lambda\in(0,1]$.} In this case, we have that 
  $\{b_\xi\}_{\xi\in\Sigma}$ converges weakly to an element $b\in B_{(X,|\!|\!|\cdot|\!|\!|)}$ (indeed, eventually $\lambda_\xi\neq0$ and hence we can write $b_\xi=\frac{x_\xi-(1-\lambda_\xi)c_\xi}{\lambda_\xi}$). Since $U_{(X,|\!|\!|\cdot|\!|\!|)}\subset \inte D$  we necessarily have $b\in S_{(X,|\!|\!|\cdot|\!|\!|)}$. Since $|\!|\!|\cdot|\!|\!|$ has the Kadec property, we have that $|\!|\!|b_\xi -b|\!|\!|\to 0$ and hence that  the net $\{x_\xi\}_{\xi\in \Sigma}$ converges in norm to $x$. 
	
	In any case, we get a contradiction and the proof is concluded.
\end{proof}

	\begin{theorem}\label{th: main theorem}
Let $X$ be a separable Banach space and 	 $|\cdot|$  defined as above. Then  the following conditions hold:
		\begin{enumerate}
			\item $(X,|\cdot|)$ is G\^ateaux smooth;
			\item $(X,|\cdot|)$ is ALUR;
	\item if $X$ is reflexive then  $(X^*,|\cdot|^*)$ is Fr\'echet smooth;
			\item $(X,|\cdot|)$ is not LUR.
		\end{enumerate}
  \end{theorem}
	
	\begin{proof}(i) It is sufficient to observe that the function $H:{X}\to \R$, defined by $$H(x)= \sum_{n=2}^\infty2^{-n}[f_n(x)]^2,\qquad x\in{X},$$ is G\^ateaux differentiable on ${X}$. By Proposition~\ref{prop: smooth}, the map $F=|\cdot|^2$ is sum of two G\^ateaux differentiable functions and hence $|\cdot|$ is G\^ateaux smooth.
		
		\noindent (ii) By Proposition~\ref{prop: rotund} and Theorem~\ref{th: ALURsseR+KK}, it is sufficient to prove that  $|\cdot|$ has the Kadec property. 
  Let $\{x_\xi\}_{\xi\in\Sigma}\subset S_{(X,|\cdot|)}$ be a net that is $w$-convergent to an element $x\in S_{(X,|\cdot|)}$.
  Then we have that, for each $n\in\N$, $\lim_\xi f_n(x_\xi)= f_n(x)$. This easily implies that the function $H$, defined as in the previous point, satisfies $\lim_\xi H(x_\xi)=H(x)$. Hence, by the definition of $|\cdot|$, we have $\lim_\xi\|x_\xi\|=\|x\|$. By Proposition~\ref{prop: KP} and the fact that $\{x_\xi\}_{\xi\in\Sigma}$  $w$-converges to $x$, we have that $\{x_\xi\}_{\xi\in\Sigma}$  $\|\cdot\|$-converges to $x$. 
		
		\noindent (iii)	The proof follows by Theorem~\ref{th: ALURsseR+KK}, and the previous point.		

 \noindent (iv)
  Let us prove that $x_0=\sqrt2 e_1\in S_{(X,|\cdot|)}$ is not a LUR point for $B_{(X,|\cdot|)}$. For $n\in \N$, define $$x_n=\frac{1}{\sqrt2}e_1+\frac{1}{\sqrt2}e_{3n},\qquad x_n^*=f_1+f_{3n}=g_1+g_{3n},$$
		and observe that
		\begin{enumerate}[(a)]
			\item $x_n\in B_{{(X,|\!|\!|\cdot|\!|\!|)}}\subset D$;
			\item $\{x_n^*\}_n$ is bounded in $X^*$;
			\item for each $n\in\N$, we have $$\sup x_n^*(B_{{(X,|\!|\!|\cdot|\!|\!|)}})=\sqrt2,\qquad\sup x_n^*(B)\leq\sqrt2+\frac1{(3n)^2}$$
			and hence  
			$$\sup x_n^*(D)\leq\sqrt2+\frac1{(3n)^2};$$
			\item for each $n\in\N$, we have $$ x_n^*(x_n)=x_n^*(x_0)=\sqrt2=x_n^*\left(\frac{x_0+x_n}{2}\right).$$
		\end{enumerate}

	\noindent Now, for $n\in\N$, define $z_n=\frac{x_0+x_n}{2}+\frac{x_n}{\sqrt2(3n)^2}$, and observe that $$x_n^*(z_n)=\sqrt2+\frac1{(3n)^2}\geq\sup x_n^*(D).$$	
 In particular, we have $\|z_n\|\geq1$ and hence there exists $w_n\in[\frac{x_0+x_n}{2},z_n]\cap\partial D$. 
Hence,
$$\textstyle\dist_{\|\cdot\|}\left(\frac{x_0+x_n}{2}, \partial D\right)\leq \|\frac{x_0+x_n}{2}-w_n\|\leq \|\frac{x_0+x_n}{2}-z_n\|\to 0,$$
  that is, $\|\frac{x_0+x_n}{2}\|\to 1$. Let us also observe that $ |x_0|=1$ and that, for each $n\in\N$, we have
		$$\textstyle |x_n|^2\leq1+2^{-3n-1},\qquad\left|\frac{x_0+x_n}{2}\right|^2\geq\left\|\frac{x_0+x_n}{2}\right\|^2;$$
  hence, since $2|x_0|^2+2|x_n|^2-|x_0+x_n|^2\geq0$, we have 
  $$\lim_n (2|x_0|^2+2|x_n|^2-|x_0+x_n|^2)=0.$$
  On the other hand, for each $n\in\N$, we have 
  $|\!|\!|e_1-e_n|\!|\!|\geq g_1(e_1-e_n)=1$,
		and hence the sequence $\{x_n\}_n$ does not converge in norm to $x_0$, the proof is concluded. 
	\end{proof}

 We conclude this section with a trivial consequence of the previous result. Recall that a Banach space is said \textit{weakly locally uniformly rotund (WLUR)} if each point of its unit sphere is WLUR (see Definition \ref{d: basic def}).
\begin{corollary}
    Every infinite-dimensional separable Banach space admits an ALUR G\^ateaux smooth equivalent norm which is not WLUR.
\end{corollary}
\begin{proof}
    It follows by observing that the norm defined in Theorem \ref{th: main theorem} satisfies the Kadec-Klee property.
\end{proof}

\section{Consequences in nonseparable Banach spaces}

The present section is devoted to lift the construction made in Section \ref{sec: main result} to the nonseparable setting. Let us first notice that if a Banach space has an ALUR norm, then it admits an equivalent LUR norm, see \cite[Proposition 1.4]{Tro85}. Therefore, not every Banach space can be renormed with an ALUR G\^ateaux smooth norm which is not LUR. On the other hand, the following result, which is an immediate consequence of Theorem \ref{th: main theorem}, provides a sufficient condition for a Banach space to have an equivalent ALUR G\^ateaux smooth norm which is not LUR. 
\begin{corollary}\label{cor:lift}
Let $(X,\|\cdot\|)$ be a Banach space and $Y\subset X$ be a separable subspace. Suppose that:
\begin{itemize}
    \item the norm $\|\cdot\|$ is LUR and G\^ateaux smooth;
    \item $Y$  is infinite-dimensional and complemented in $X$.
\end{itemize}
Then $X$ admits an equivalent norm which is ALUR, G\^ateaux smooth but not LUR.
\end{corollary}
\begin{proof}
    By Theorem \ref{th: main theorem} there exists an ALUR, G\^ateaux smooth, not LUR norm $|\cdot|_Y$ on $Y$ which is equivalent to the norm $\|\cdot\|$ restricted to $Y$. Let $P:X\to X$ be a bounded linear projection onto $Y$. The norm $|\cdot|$, defined for $x\in X$ by
\begin{equation*}
    |x|^2=\|(I-P)x\|^2 + |Px|_Y^2,
\end{equation*}
     satisfies the desired properties.
\end{proof}

It is worth to notice that the hypotheses of Corollary~\ref{cor:lift} are satisfied by a wide class of Banach spaces. To clarify this, let us recall the following definitions (see, e.g., \cite[Section~6.1.6]{GMZ}).

\begin{definition}\label{def: WCGandsons}
Let $X$ be a Banach space. Then $X$ is said:
\begin{enumerate}
    \item {\em weakly compactly
generated} (WCG), if there exists a weakly compact subset $K$
of $X$ such that $\overline{\mathrm{span}} (K) = X$;
\item   {\em weakly countably
determined} (WCD) if there is a countable collection $\{K_n; n\in\N\}$ of
$w^*$-compact subsets of $X^{**}$ such that for every $x\in X$ and $u\in X^{**}\setminus X$ there is
$n_0\in\N$ for which $x\in K_{n_0}$ and $u\not\in K_{n_0}$;
\item 
{\em weakly Lindel\"of determined} (WLD)
if there exist a set $\Gamma$ and a one-to-one bounded linear operator from $X^*$ into
$\ell_\infty^c(\Gamma)$ that is $w^*$-pointwise-continuous, where $\ell_\infty^c(\Gamma)$ denotes the subspace of $\ell_\infty(\Gamma)$ formed by all elements countably supported;
\item to have the {\em separable complementation property} (SCP) if  every separable subspace of $X$ is contained in a 
complemented separable subspace of $X$.
\end{enumerate}
\end{definition}

\noindent The following implications hold (see \cite{GMZ}):
\medskip

\begin{center}
\begin{tabular}{cccccccccc}
&  & $\langle\text{R}^*\rangle$ &  &   $\langle\text{LUR}\rangle$ & & \\
        &               &      $\Uparrow$  &              &                     $\Uparrow$ &    & \\
 (WCG) & $\Rightarrow$ & (WCD) & $\Rightarrow$ & (WLD) & $\Rightarrow$& (SCP)     
  \end{tabular}
\end{center}
\medskip

\noindent where   $\langle\text{LUR}\rangle$
 means that 
 the space $X$  admits an equivalent LUR norm, and $\langle\text{R}^*\rangle$ means that the space  $X$ admits an equivalent norm such that its dual norm is rotund. Notice that none of the above implications can be reversed.

Let us recall that, by the Asplund averaging method (see \cite[Theorem~4.1]{DGZ}), if $X$ has $\langle\text{R}^*\rangle$ and  $\langle\text{LUR}\rangle$ then it admits an equivalent LUR norm such that its dual norm is rotund. Moreover, if a Banach space $X$ has the (SCP) then, in particular, it contains a complemented infinite-dimensional separable subspace. Hence, it is clear that Corollary~\ref{cor:lift} can be applied to the class (WLD) spaces having $\langle\text{R}^*\rangle$. In particular, {\em each (WCD) space admits an equivalent norm which is ALUR, G\^ateaux smooth but not LUR}.
Taking into account Theorems~\ref{th: ALURsseR+KK}~and~\ref{th: ALUR vs almost Lur}, we clearly obtain the following corollary.
\begin{corollary} Let $X$ be a reflexive space. Then $X$ admits an equivalent  G\^ateaux smooth not LUR norm $|\cdot|$  such that  $|\cdot|^*$ is Fr\'echet smooth.
    \end{corollary}
    
The Banach spaces of continuous functions $C([0,\alpha])$, with $\alpha$ uncountable ordinal, constitute another relevant class of nonseparable Banach spaces to which we can apply Corollary~\ref{cor:lift}, and that is not contained in the class of (WLD) spaces. Indeed, the containment of an infinite-dimensional separable complemented subspace is trivial and existence of an equivalent norm that is LUR and  G\^ateaux smooth follows by  \cite{HP}.

\subsection*{Acknowledgements} We are grateful to the anonymous referee for a careful reading of the manuscript and for several suggestions which improved the presentation.

\end{document}